%% file: root.tex
\pdfoutput=1
\documentclass[letterpaper, 10 pt, conference]{ieeeconf}  

\IEEEoverridecommandlockouts                              

\usepackage{mathptmx} 
\usepackage{amsmath} 
\usepackage{amssymb}  
\usepackage{mathrsfs}
\usepackage{cleveref}
\usepackage{float}
\usepackage{svg}
\usepackage{graphicx}
\graphicspath{{./figures/}}
\crefformat{equation}{(#2#1#3)}

\newtheorem{problem}{Problem}

\newtheorem{theorem}{Theorem}

\newtheorem{assumption}{Assumption}
\newtheorem{lemma}{Lemma}
\newtheorem{note}{Remark}

\DeclareMathOperator{\rank}{rank\,}

\DeclareMathOperator{\nul}{null\,}

\newcommand{\norm}[1]{\left\lVert#1\right\rVert}

\title{\LARGE \bf
Data-Driven Output Matching of Output-Generalized Bilinear and Linear Parameter-Varying systems
}

\author{Leander Hemelhof$^{1}$, Ivan Markovsky$^{2}$ and Panagiotis Patrinos$^{3}$
\thanks{*This work is supported by: the Research Foundation Flanders (FWO) research projects G081222N, G033822N, G0A0920N; European Union’s Horizon 2020 research and innovation programme under the Marie Skłodowska-Curie grant agreement No. 953348;
Research Council KU Leuven C1 project No. C14/18/068; Fonds de la Recherche Scientifique – FNRS and the Fonds Wetenschappelijk Onderzoek – Vlaanderen under EOS project no 30468160 (SeLMA). FNRS--FWO EOS Project 30468160. IM is funded by the Catalan Institution for Research and Advanced Studies (ICREA).}
\thanks{$^{1,3}$L. Hemelhof and P. Patrinos are with the Department of Electrical Engineering (ESAT-STADIUS), KU Leuven, Kasteelpark Arenberg 10, 3001 Leuven, Belgium. Email: {\tt\small \{leander.hemelhof, panos.patrinos\}@esat.kuleuven.be}}%
\thanks{$^{2}$Ivan Markovsky is with the the International Centre for Numerical Methods in Engineering (CIMNE), Gran Capitàn, Barcelona, Spain and the Catalan Institution for Research and Advanced Studies (ICREA), Pg. Lluis Companys 23, Barcelona, Spain. Email: {\tt\small imarkovsky@cimne.upc.edu}}%
}

\begin{document}

\maketitle
\thispagestyle{empty}
\pagestyle{empty}

\input{main}

\bibliographystyle{IEEEtran}
\bibliography{IEEEabrv,all}
\end{document}

%% file: main.tex
\begin{abstract}
    There is a growing interest in data-driven control of nonlinear systems over the last years. In contrast to related works, this paper takes a step back and aims to solve the output matching problem, a problem closely related to the reference tracking control problem, for a broader class of nonlinear systems called output-generalized bilinear, thereby offering a new direction to explore for data-driven control of nonlinear systems. It is shown that discrete time linear parameter-varying systems are included in this model class, with affine systems easily shown to also be included. This paper proposes a method to solve the output matching problem and offers a way to parameterize the solution set with a minimal number of parameters. The proposed model class and method are illustrated using simulations of two real-life systems.
\end{abstract}
\section{Introduction}
Data-driven simulation and control of nonlinear systems has been a popular topic over the last years. Some results have already been obtained for specific classes of nonlinear systems. A few examples are simulation of Wiener and Hammerstein systems \cite{berberich_trajectory-based_2020}, control of linear parameter-varying systems with affine dependence on the parameters \cite{verhoek_fundamental_2021,verhoek_data-driven_2021}, simulation of generalized bilinear systems \cite{markovsky_data-driven_2023}, control of second-order discrete Volterra systems \cite{rueda-escobedo_data-driven_2020}, control of discrete time MIMO feedback linearizable systems \cite{alsalti_data-based_2022} and control of flat nonlinear systems \cite{alsalti_data-based_2021}. More general control results also exist, for example making use of local linear approximations \cite{berberich_linear_2022}, kernels \cite{huang_robust_2022,tanaskovic_data-driven_2017} or the Koopman operator \cite{arbabi_data-driven_2018,lian_koopman_2021}.

In this work we solve the output matching problem for a broad class of nonlinear systems. The output matching problem was originally discussed in \cite{markovsky_data-driven_2008} while building up to reference tracking control, showing its usefulness as a stepping stone to more practically useful control problems. Since the output matching problem is closely related to the reference tracking control problem, this work offers an additional direction to explore for data-driven control of nonlinear systems. The main contributions of this paper are the following:
\begin{enumerate}
    \item We define a novel model class that encompasses multiple other model classes discussed in the literature, among which linear parameter-varying systems with known parameter values and affine dependence on the parameters and affine systems. 
    \item We provide a method to solve the output matching problem for this model class in the noiseless case, offering a parameterization of the solution space.
\end{enumerate}

\Cref{sec2} introduces notation and definitions necessary in the following sections. \Cref{sec3} introduces the output-generalized bilinear model class and solves the output matching problem for this model class. \Cref{sec4} shows that linear parameter-varying systems with known parameter values are a specific case of output-generalized bilinear systems, allowing the results from \cref{sec3} to apply to them as well. \Cref{sec5} gives some illustrative simulation examples, while \cref{sec6} gives some more general remarks about the parameterization of the solution space and the minimal data trajectory length. \Cref{sec7} concludes this work and offers some final remarks regarding future work.

\section{Preliminaries}\label{sec2}
In the behavioral framework \cite{willems1991} we define a discrete time dynamical system $\mathcal{B}$ as a set of trajectories. We write for example $\mathcal{B}\subset(\mathbb{R}^q)^\mathbb{N}$, where $(\mathbb{R}^q)^\mathbb{N}$ is the set of $q$-variate signals. $\mathcal{B}|_{L}$ is the subset containing all possible trajectories of length $L$ of the system. $\mathcal{L}^q$ is the set of \textit{linear time-invariant} (LTI) models with $q$ variables. We define the shift operator $\sigma$ as $\sigma y(t)=y(t+1)$. The kronecker product of two matrices $A$ and $B$ is denoted by $A\otimes B$. We define the concatenation operator as $w_1\wedge w_2=\begin{bmatrix}
w_1^\top & w_2^\top
\end{bmatrix}^\top$, with $w_1\in(\mathbb{R}^{q_1})^{T_1}$ and $w_2\in(\mathbb{R}^{q_2})^{T_2}$. The block Hankel matrix of a trajectory $w\in(\mathbb{R}^{q})^{T}$ with $L$ block rows is denoted by $\mathcal{H}_L(w)$, with
\begin{equation*}
    \resizebox{\linewidth}{!}{$\mathcal{H}_L(w)=\begin{bmatrix}
    w(1) & w(2) & \cdots & w(T-L+1)\\
    w(2) & w(3) & \cdots & w(T-L)\\
    \vdots & \vdots & & \vdots\\
    w(L) & w(L+1) & \cdots & w(T)
    \end{bmatrix}\in\mathbb{R}^{qL\times (T-L+1)}.$}
\end{equation*}

Formally, the data-driven output matching problem can be formulated in the behavioral framework as follows \cite[Problem 13]{markovsky_data-driven_2008}:
\begin{problem}[Data-driven output matching]
For a system $\mathcal{B}$ and with
\begin{itemize}
    \item a trajectory $w_{\mathrm{d}}\in(\mathbb{R}^q)^T$ of $\mathcal{B}|_{T}$,
    \item an input-output partitioning $Pw=\begin{bmatrix}
    u\\y
    \end{bmatrix}$ with $P$ a permutation matrix and $q=n_u+n_y$,
    \item a reference output $y_\mathrm{r}\in(\mathbb{R}^{n_y})^L$,
    \item an initial trajectory $w_{\mathrm{ini}}\in(\mathbb{R}^{q})^{T_{\mathrm{ini}}}$ of $\mathcal{B}|_{T_{\mathrm{ini}}}$,
\end{itemize}
find an input $u\in(\mathbb{R}^{n_u})^L$ such that $w_{\mathrm{ini}}\wedge (u,y_\mathrm{r})\in\mathcal{B}|_{T_{\mathrm{ini}}+L}$.\label{prob1}
\end{problem}

This problem does not always have a solution. Consider for example an LTI system, defined by $y(t)=0.5y(t-1)+u(t-1)$. If we give a reference output $y_\mathrm{r}=\begin{bmatrix}
2 & 2
\end{bmatrix}$ while $u_{\mathrm{ini}}=0$ and $y_{\mathrm{ini}}=5$, no matter what input we give to the system, it will be impossible to match the reference output. To deal with this, a closely related problem that always has a solution is given by the reference tracking control problem:
\begin{problem}[Data-driven reference tracking control]
For a system $\mathcal{B}$ and with
\begin{itemize}
    \item a trajectory $w_{\mathrm{d}}\in(\mathbb{R}^q)^T$ of $\mathcal{B}|_{T}$,
    \item an input-output partitioning $Pw=\begin{bmatrix}
    u\\y
    \end{bmatrix}$ with $P$ a permutation matrix and $q=n_u+n_y$,
    \item a reference output $y_\mathrm{r}\in(\mathbb{R}^m)^L$,
    \item an initial trajectory $w_{\mathrm{ini}}\in(\mathbb{R}^q)^{T_{\mathrm{ini}}}$ of $\mathcal{B}|_{T_{\mathrm{ini}}}$,
    \item a cost function $\mathcal{C}(u, y, y_\mathrm{r})$,
\end{itemize}
find an input $u\in(\mathbb{R}^{n_u})^L$ and output $y\in(\mathbb{R}^{n_y})^L$ such that $w_{\mathrm{ini}}\wedge (u,y)\in\mathcal{B}|_{T_{\mathrm{ini}}+L}$\label{prob2} and that minimizes cost function $\mathcal{C}(u, y, y_\mathrm{r})$.
\end{problem}

This paper deals with the considered nonlinear systems by using LTI embedding, where we embed the nonlinearity into an LTI system. This allows us to use results meant for LTI systems. With the goal of designing a data-driven method in mind, we will make use of the data-driven representation of LTI systems. This representation is given by following lemma \cite[Corollary 3]{markovsky_data-driven_2023}:

\begin{lemma}[Data-driven representation]
Consider $\mathcal{B}\in\mathcal{L}^q$, data trajectory $w_{\mathrm{d}}\in\mathcal{B}|_T$ and $L\geq \textbf{l}(\mathcal{B})$. Assuming
\begin{equation}
    \rank \mathcal{H}_L(w_{\mathrm{d}})=\textbf{m}(\mathcal{B})L+\textbf{n}(\mathcal{B}),\label{dd_cond}
\end{equation} then $w\in\mathcal{B}|_L$ if and only if there exists $g\in\mathbb{R}^{T-L+1}$ such that\begin{equation}
    \mathcal{H}_L(w_{\mathrm{d}})g=w.
\end{equation}
\label{dd_rep}
\end{lemma}
Here $\textbf{l}(\mathcal{B})$, $\textbf{m}(\mathcal{B})$ and $\textbf{n}(\mathcal{B})$ are the lag, the number of inputs and the order of the system respectively. They are all properties of the system \cite{willems1991}. Note that since the number of inputs is a property of the system, it is independent from the chosen input-output partitioning.

\section{Output-generalized bilinear systems}\label{sec3}
We consider a MIMO system that can be written in the form
\begin{align*}
    y(t)=f(x_y(t), x_{h(u)}(t),t)=\ &\theta_{\mathrm{lin}}^\top\begin{bmatrix}
    x_\mathrm{y}(t)^\top & x_{h(u)}(t)^\top
    \end{bmatrix}^\top\\
    &+\theta_{\mathrm{nl}}^\top\phi_{\mathrm{nl}}(x_y(t), x_{h(u)}(t), t),
\end{align*}
where
\begin{align*}
    &\phi_{\mathrm{nl}}(x_y(t), x_{h(u)}(t),t)=\phi_{\mathrm{b},0}(x_y(t), t)+\phi_\mathrm{b}(x_y(t), t)\otimes x_{h(u)}(t),\\
    &x_y(t)=\begin{bmatrix}y(t-l)^\top & y(t-l+1)^\top & \hdots & y(t-1)^\top\end{bmatrix}^\top,\\
    &\resizebox{\linewidth}{!}{$x_{h(u)}(t)=\begin{bmatrix}h(u(t-l))^\top & h(u(t-l+1))^\top & \hdots & h(u(t))^\top\end{bmatrix}^\top,$}
\end{align*}
for all $t\in\mathbb{Z}$. Here time is allowed to be negative to capture the influence of the infinite past of the system. In practice we have signals such that $t\in\mathbb{N}$. To ease notation, from this point on we will not write the time dependence explicitly. We only have to remember that $\phi_{\mathrm{nl}}$ can depend on time. This way we obtain
\begin{equation}
    y=f(x_y, x_{h(u)})=\theta_{\mathrm{lin}}^\top\begin{bmatrix}
    x_\mathrm{y}^\top & x_{h(u)}^\top
    \end{bmatrix}^\top+\theta_{\mathrm{nl}}^\top\phi_{\mathrm{nl}}(x_y, x_{h(u)})\label{system1},
\end{equation}
where
\begin{align}
    \phi_{\mathrm{nl}}(x_y, x_{h(u)})&=\phi_{\mathrm{b},0}(x_y)+\phi_\mathrm{b}(x_y)\otimes x_{h(u)},\label{OGB_phi_nl}\\
    x_y&=\begin{bmatrix}\sigma^{-l} y^\top & \sigma^{-l+1} y^\top & \hdots & \sigma^{-1} y^\top\end{bmatrix}^\top,\nonumber\\
    x_{h(u)}&=\begin{bmatrix}\sigma^{-l} h(u)^\top & \sigma^{-l+1}h(u)^\top & \hdots & h(u)^\top\end{bmatrix}^\top.\nonumber
\end{align}
Here $\theta_{\mathrm{lin}}\in\mathbb{R}^{l(n_\mathrm{y}+n_\mathrm{u})+n_\mathrm{u}\times n_\mathrm{y}}$ and $\theta_{\mathrm{nl}}\in\mathbb{R}^{n_\mathrm{b}\times n_\mathrm{y}}$ are constant parameter matrices and $h: \mathbb{R}^{n_u}\rightarrow\mathbb{R}^{n_u}$ is a known bijective map with a known inverse. Functions $\phi_{\mathrm{b}}$ and $\phi_{\mathrm{b,0}}$ are possibly time dependent nonlinear functions of the form $\phi:\mathbb{R}^{n_\mathrm{y}l}\times\mathbb{R}\rightarrow\mathbb{R}^{n_\mathrm{b}}$, where the extra dimension of the input space indicates the time dependence. The main use case of this explicit time dependence is to allow the nonlinear terms to depend on other known signals, which are interpreted here as time-dependent functions. This will be further illustrated later in this section.

A similar SISO model class, called \textit{generalized bilinear}, is introduced in \cite{markovsky_data-driven_2023}. It has the same form as used here, but the nonlinear term is given by \begin{equation*}
    \phi_{\mathrm{nl}}(x)=\phi_{\mathrm{b},0}+\phi_\mathrm{b}(x_u)\otimes x_y.
\end{equation*}
In contrast to this generalized bilinear model class we refer to our model class as output-generalized bilinear.

The output-generalized bilinear model class is given by the sum of some term that depends linearly on the previous outputs and current and previous inputs, and a nonlinear term. The nonlinear term is a generalization of bilinear systems that allows for nonlinear dependence on the previous outputs, mixed with current and previous inputs. More specifically, it consists of terms in function of the previous outputs (and possibly known functions of time), some multiplied with current or previous inputs. The inputs are allowed to be transformed using a known bijective transformation. All nonlinear functions are assumed known up to finite sets of basis functions, whose span contains the nonlinear function. To ease notation we will write $h(u(t))=u_\mathrm{h}(t)$. To be able to do data-driven output matching, the main requirement on the system is that its nonlinearity needs to be affine in time shifts of some bijective transformation of the input. At this point we only make one assumption on the nonlinear functions $\phi_{\mathrm{b}}$ and $\phi_{\mathrm{b,0}}$:
\begin{assumption}
    We have a finite set of basis functions such that both nonlinear functions $\phi_{\mathrm{b}}$ and $\phi_{\mathrm{b,0}}$ can be written as a linear combination of functions from this set.\label{ass1}
\end{assumption}

We don't make any assumptions regarding the properties of the functions themselves. In particular, they don't need to be differentiable or even continuous.

In some cases it is convenient to use a slightly altered definition of the nonlinear term compared to \cref{OGB_phi_nl}:
\begin{equation}
    \phi_{\mathrm{nl}}(x)=\phi_{\mathrm{b},0}+\begin{bmatrix}
    \phi_\mathrm{b}(x_\mathrm{y})\otimes x_{h(u)}\\\mathbf{0}_{k\times 1}
    \end{bmatrix},\label{OGB_phi_nl2}
\end{equation}
where $\phi_{\mathrm{b,0}}$ is re-scaled accordingly. This definition makes it easier to rewrite other model classes with nonlinear terms independent from the inputs into form \cref{system1}. Since the effect of adding these zero rows could also be obtained by adding some zero rows to $\phi_\mathrm{b}$ and $\phi_\mathrm{b,0}$, this alternative definition doesn't change anything functionally and only offers convenience.

Since this is a complicated model class, some examples are given to illustrate its use. Unless otherwise specified, the bijective input transformation is taken as the identity transformation $h(u)=u$.
\begin{itemize}
    \item \textit{LTI system}. Consider an double input single output LTI system with lag 2 given in input-output form:
    \begin{equation*}
        y(t)=\sum_{i=0}^{1}\theta_\mathrm{y}^\top(i)y(t-2+i)+\sum_{j=0}^2\theta_\mathrm{u}^\top(j)u(t-2+j).
    \end{equation*}
    This can be written into form \cref{system1} by taking $\theta_{\mathrm{lin}}=\begin{bmatrix}
    \theta_{\mathrm{y}}^\top & \theta_{\mathrm{u}}^\top
    \end{bmatrix}^\top$, $\theta_{\mathrm{nl}}=0$ and $\phi_{\mathrm{b,0}}(x_y)=\phi_{\mathrm{b}}(x_y)=0$. The next examples will build on this example, so to aid notation, denote the values corresponding to this example with superscript LTI.
    \item \textit{LTI system with nonlinear terms}. Consider a nonlinear system given by the sum of an LTI system and some nonlinear terms:
    \begin{equation*}
        y(t) = y^{\mathrm{LTI}}(t)-1.5\mathrm{sin}(y(t-1))+0.3y(t-2)^2u_2(t-1).
    \end{equation*}
    The linear part stays the same, so $\theta_{\mathrm{lin}}=\theta_{\mathrm{lin}}^{\mathrm{LTI}}$. Making use of \cref{OGB_phi_nl2}, we start off by looking at the nonlinear term that depends on the input. We have $\phi_\mathrm{b}(x_y)=y(t-2)^2$. The kronecker product makes it so that we get $n_\mathrm{u}(l+1)=6$ nonlinear terms from this one function, of which only one is of interest. The other terms will just get a zero value in the parameter matrix. We also have a term that doesn't depend on the input, such that $\phi_{\mathrm{b,0}}(x_y)=\begin{bmatrix}
    \textbf{0}_{6\times 1}^\top & \mathrm{sin}(y(t-1))
    \end{bmatrix}^\top$ and we end up with $n_\mathrm{b}=7$ nonlinear terms, for once written explicitly:
    \begin{align*}
        \phi_{\mathrm{nl}}=\big[&y(t-2)^2u_1(t-2),\  y(t-2)^2u_2(t-2),\\
        &y(t-2)^2u_1(t-1),\  y(t-2)^2u_2(t-1),\\
        &y(t-2)^2u_1(t),\  y(t-2)^2u_2(t),\  \mathrm{sin}(y(t-1))\big]^\top.
    \end{align*}
    With this it is easy to see that
    \begin{equation*}
        \theta_{\mathrm{nl}}=\begin{bmatrix}
        0 & 0 & 0 & 0.3 & 0 & 0 & -1.5
        \end{bmatrix}^\top.
    \end{equation*}
    \item \textit{LTI system with nonlinear term depending on other known signal}. Assume we have a signal $p(t)$ of which we know its values for all $t_0\in\mathbb{N}<t$. This can be due to knowing its generating function or having measured it. Consider a nonlinear system given by the sum of an LTI system and a signal-dependent term:
    \begin{equation*}
        y(t) = y^{\mathrm{LTI}}(t)+0.8(p(t-1)y(t-2))^3u_1(t).
    \end{equation*}
    The linear part stays the same, so $\theta_{\mathrm{lin}}=\theta_{\mathrm{lin}}^{\mathrm{LTI}}$. Being a signal, $p(t)$ can also be interpreted as a discrete function of time $p:\mathbb{Z}\rightarrow\mathbb{P}$ with $\mathbb{P}$ the subspace of possible parameter values. This discrete function of time can be used in places where an explicit time dependence is allowed, like the $\phi$ functions. Under our assumptions this function is known since we know the signal where needed. Similar to the previous example we obtain $\phi_{\mathrm{b}}(x_y)=(p(t-1)y(t-2))^3$ and $\phi_{\mathrm{b,0}}(x_y)=0$. We get
    \begin{equation*}
        \theta_{\mathrm{nl}}=\begin{bmatrix}
        0 & 0 & 0 & 0 & 0.8 & 0
        \end{bmatrix}^\top.
    \end{equation*}
    \item \textit{nonlinear system with bijective transformation of the input}. Consider a system given by
    \begin{equation*}
        \resizebox{\linewidth}{!}{$y(t)=0.5y(t-1)+0.3(u_1(t-1)+u_2(t-1)^2)+u_2(t-2)^3.$}
    \end{equation*}
    If we take $h(u(t))=u_h(t)=\begin{bmatrix}
    u_1(t)+u_2(t)^2 & u_2(t)^3
    \end{bmatrix}^\top$, we obtain following system:
    \begin{equation*}
        y(t)=0.5y(t-1)+0.3u_{h,1}(t-1)+u_{h,2}(t-2),
    \end{equation*}
    which is a linear system. It is easy to see that
    \begin{equation*}
        \theta_{\mathrm{lin}}=\begin{bmatrix}
        0 & 1 & 0 & 1 & 0.3 & 0 & 0 & 0
        \end{bmatrix}^\top.
    \end{equation*}
    We also have $\theta_{\mathrm{nl}}=0$ and $\phi_{\mathrm{b,0}}(x_y)=\phi_{\mathrm{b}}(x_y)=0$. The inverse of $h(u)$ is given by $u=\begin{bmatrix}
    u_{h,1}-u_{h,2}^{2/3} & u_{h,2}^{1/3}
    \end{bmatrix}$, which shows that this is indeed a bijective transformation.
\end{itemize}

A few classes of nonlinear systems are easily shown to be included in this model class. Affine systems can be included by taking $\phi_{\mathrm{b,0}}(x_y)=1$. The subclass of Hammerstein systems with bijective input transformation is easily included by taking $h$ as this bijective transformation.
\begin{note}
In theory, the bijective mapping $h(u(t))$ is allowed to depend on previous inputs as if they were known constants. This is allowed since at the time we need to invert the transformation of $u_\mathrm{h}(t)$ we already know the inputs at the time instances before it, which makes them known constants. The only place this could make things difficult is in a control setting with input constraints, since transforming these constraints to $u_\mathrm{h}$ from $u$ becomes non-trivial. Allowing for this dependence on previous time samples also makes it so that we need to know more inputs in the past compared to when this is not the case.
\end{note}

Similar to \cite{markovsky_data-driven_2023}, we start by defining an additional input $u_{\mathrm{nl}}=\phi_{\mathrm{b},0}(x_y)+\phi_b(x_y)\otimes x_{h(u)}$ such that 
\begin{equation}
    \mathcal{B}_{\mathrm{ext}}:=\left\{w_{\mathrm{ext}}=\begin{bmatrix}w\\u_{\mathrm{nl}}\end{bmatrix} \ \vline \ y=\theta_{\mathrm{lin}}^\top\phi_{\mathrm{lin}}(x_w)+\theta_{\mathrm{nl}}^\top x_{\mathrm{nl}}\right\},\label{OGB_general}
\end{equation}
where 
\begin{align*}
    x_{\mathrm{nl}}&=\begin{bmatrix}\sigma^{-l} u_{\mathrm{nl}}^\top & \sigma^{-l+1} u_{\mathrm{nl}}^\top & \hdots & u_{\mathrm{nl}}^\top\end{bmatrix}^\top\\
    \theta_{\mathrm{nl}}&\in\mathbb{R}^{(l+1)n_{\mathrm{nl}}\times n_\mathrm{y}}
\end{align*}
is an LTI system with $n_u+n_{\mathrm{nl}}$ inputs. To formulate the main result of this section we need to define some notation. Define following partitioning of the extended trajectory block Hankel matrix:

\begin{equation*}
    \mathcal{H}_{L+T_{\mathrm{ini}}}(w_{\mathrm{ext}})=\begin{bmatrix}\mathcal{H}_{L+T_{\mathrm{ini}}}(u_\mathrm{h})\\\mathcal{H}_{L+T_{\mathrm{ini}}}(y)\\\mathcal{H}_{L+T_{\mathrm{ini}}}(u_{\mathrm{nl}})\end{bmatrix}=\begin{bmatrix}U_\mathrm{h}\\Y\\U_{\mathrm{nl}}\end{bmatrix}.
\end{equation*}
For ease of notation, these matrices will be further partitioned into past and future, which gives by making use of \Cref{dd_rep}:
\begin{equation}
    \mathcal{H}_{L+T_{\mathrm{ini}}}(w_{\mathrm{ext}})g=\begin{bmatrix}U_\mathrm{h}\\Y\\U_{\mathrm{nl}}\end{bmatrix}g=\begin{bmatrix}U_{\mathrm{hp}}\\U_{\mathrm{hf}}\\Y_p\\Y_f\\U_{\mathrm{nl}}\end{bmatrix}g=\begin{bmatrix}u_{\mathrm{h,ini}}\\u_\mathrm{h}\\y_{\mathrm{ini}}\\y_{\mathrm{f}}\\u_{\mathrm{nl}}\end{bmatrix}.\nonumber
\end{equation}

Another useful formula is one that allows us to write $u_{\mathrm{nl}}$ linearly in function of the inputs:
\begin{equation}
    u_{\mathrm{nl}}=\phi_{\mathrm{b},0}+\Phi_{\mathrm{b},p}(x_y)u_{\mathrm{h,ini}}+\Phi_{\mathrm{b},f}(x_y)u_\mathrm{h}.\label{unl_lin}
\end{equation}
The constructive proof of this formula is similar to the one of \cite[Lemma 6]{markovsky_data-driven_2023}.
\begin{lemma}[Matrix representation of \cref{OGB_phi_nl}]
For an output-generalized bilinear system $\mathcal{B}$ and a finite signal $w\in(\mathbb{R}^q)^T$ with $T>\mathbf{l}(\mathcal{B})$, it holds that $\phi_{\mathrm{b,0}}=\phi_{\mathrm{nl}}(x_\mathrm{y},\mathbf{0})\in\mathbb{R}^{n_{\mathrm{nl}}n_\mathrm{u}(T-\mathbf{l})}$ and $\phi_{\mathrm{b}}(x_\mathrm{y})\otimes x_{h(u)}=\Phi_\mathrm{b} u_h$, where $\Phi_\mathrm{b}\in\mathbb{R}^{n_{\mathrm{nl}}n_\mathrm{u}(T-\mathbf{l})\times n_\mathrm{u}T}$ is given by
\begin{equation}
    \resizebox{\linewidth}{!}{$
    \Phi_\mathrm{b}(x_y)=\begin{bmatrix}
    \phi_\mathrm{nl}\left(x_y,x(\delta)\right)-\phi_{\mathrm{b,0}}(x_y) & \cdots & \phi_\mathrm{nl}\left(x_y,x(\sigma^{n_uT-1}\delta)\right)-\phi_{\mathrm{b,0}}(x_y)
    \end{bmatrix}$}\label{Phib}
\end{equation}
with $\delta$ the unit pulse $\left(1,0,\cdots,0\right)\in\mathbb{R}^{n_uT}$ and $x(\delta)$ being $x_{h(u)}$ with all $h(u)$ replaced by $\delta$. With some abuse of notation we interpret $\delta\in\mathbb{R}^{n_uT}$ as $\delta\in(\mathbb{R}^{n_u})^{T}$, allowing it to take the place of $h(u)$ where needed.
\end{lemma}
\begin{proof}
It is evident to see that taking all $u_h=0$ results in $\phi_\mathrm{nl}(x_y, \mathbf{0})=\phi_\mathrm{b,0}(x_y)$, since that is the only term remaining in that case. Taking a look at $\phi_{\mathrm{b}}(x_y)\otimes x_{h(u)}$, we see that this term is linear in $x_{h(u)}$ since each row is just a linear combination of the transformed inputs. We can extract the coefficients of each input for each row by setting that specific input to one and the rest to zero. This means that the input element at index $k$ has to be replaced by $\sigma^{k-1}\delta$. We find \cref{Phib} by noticing that $\phi_{\mathrm{b}}(x_y)\otimes x_{h(u)}=\phi_\mathrm{nl}(x_y,x_{h(u)})-\phi_\mathrm{b,0}(x_y)$.
\end{proof}

We now have the following theorem:
\begin{theorem}[data-driven output matching 1]
Consider an output-generalized bilinear system $\mathcal{B}$, data $w_{\mathrm{d}}\in\mathcal{B}|_T$ and initial trajectory $w_{\mathrm{ini}}=\begin{bmatrix}
u_{h,\mathrm{ini}}^\top & y_{\mathrm{ini}}^\top
\end{bmatrix}^\top\in(\mathbb{R}^{n_{\mathrm{u}}+n_{\mathrm{y}}})^{T_{\mathrm{ini}}}$. Under the assumption that \begin{equation}
    \rank\mathcal{H}_{L+T_{\mathrm{ini}}}(w_{\mathrm{d,ext}})=(n_\mathrm{u}+n_{\mathrm{nl}})(L+T_{\mathrm{ini}})+n,\label{poe1}
\end{equation} with $n$ the order of the system, and \Cref{ass1}, a solution to \Cref{prob1} is given by
\begin{equation}
    u=h^{-1}\left(U_{\mathrm{hf}}\mathcal{A}^\dagger(w_{\mathrm{d}})\begin{bmatrix}w_{\mathrm{ini}}\\y_\mathrm{r}\\\Phi_{\mathrm{b,p}}(x_y)u_{\mathrm{h,ini}}+\phi_{\mathrm{b,0}}\end{bmatrix}\right),\label{sol1}
\end{equation}
where \begin{equation}
    \mathcal{A}(w_{\mathrm{d}})=\begin{bmatrix}U_{\mathrm{hp}}\\Y\\U_{\mathrm{nl}}-\Phi_{\mathrm{b,f}}(x_y)U_{\mathrm{hf}}\end{bmatrix},\nonumber
\end{equation}
and $\cdot^\dagger$ indicates the pseudo-inverse. Furthermore, the solution set of $u_\mathrm{h}$ can be parameterized using $\rank U_{\mathrm{hf}}N$ parameters according to \Cref{param}, with $N$ a basis for the null space of $\mathcal{A}(w_\mathrm{d})$.\label{theo1}
\end{theorem}

\begin{proof}
Using \Cref{ass1} we can use the finite set of basis functions instead of $\phi_\mathrm{b}$ and $\phi_{\mathrm{b,0}}$ themselves, which we don't assume to know. This allows us to construct $U_\mathrm{nl}$ and $\Phi_\mathrm{b}$. Under the rank condition \cref{poe1} we have, according to \Cref{dd_rep}, that there exists a $g$ such that \begin{equation}
    \mathcal{H}_{L+T_{\mathrm{ini}}}(w_{\mathrm{d,ext}})g=\begin{bmatrix}U_{\mathrm{hp}}\\U_{\mathrm{hf}}\\Y_p\\Y_f\\U_{\mathrm{nl}}\end{bmatrix}g=\begin{bmatrix}u_{\mathrm{h,ini}}\\u_\mathrm{h}\\y_{\mathrm{ini}}\\y_{\mathrm{r}}\\u_{\mathrm{nl}}\end{bmatrix}.\label{eq}
\end{equation}

Plugging the equation $u_\mathrm{h}=U_{\mathrm{hf}}g$ into \cref{unl_lin}, and plugging in the resulting expression for $u_{\mathrm{nl}}$ in \cref{eq}, we can move the term depending on $g$ to the other side, which gives, putting aside the equations we don't know the solutions of:
\begin{equation}
     \begin{bmatrix}U_{\mathrm{hp}}\\Y\\U_{\mathrm{nl}}-\Phi_{\mathrm{b},f}(x_y)U_{\mathrm{hf}}\end{bmatrix}g=\mathcal{A}(w_{\mathrm{d}})g=\begin{bmatrix}w_{\mathrm{ini}}\\y_\mathrm{r}\\\Phi_{\mathrm{b},p}(x_y)u_{\mathrm{h,ini}}+\phi_{\mathrm{b},0}\end{bmatrix}.\nonumber
\end{equation}

Solving this equation for $g$ by making use of the pseudo-inverse $\mathcal{A}^\dagger$, and by making use of the equations we put aside we obtain \cref{sol1}. The parameterization result follows directly from the constructive proof of \Cref{param}.

\end{proof}


With this theorem it is shown that it is possible to solve the output matching problem for this specific class of nonlinear systems. Furthermore, instead of just giving a solution it offers a way to obtain a parameterization of the complete solution set. Sadly, it seems to be difficult to extrapolate this result to a convex solution of the reference tracking control problem, since this method relies on the fact that the output is exactly known in advance. The solved optimization problem depends on the system output in a generally non-convex way, which makes solving the reference tracking control problem a non-convex problem in general. Finding a convex reformulation or testing the feasibility of non-convex solving methods is left for future work.

An important thing to note regarding the additional inputs is that some care has to be taken regarding the basis functions. Since the output of an LTI system with lag $l$ depends not only on the current input, but in general also the previous $l$ inputs, we need to avoid having basis functions that are time-shifted versions of other basis functions. If this is not the case we will be unable to fulfill the generalized persistency of excitation condition \cref{poe1}. To fulfill this condition we need independent inputs, including the additional inputs, which is not the case if one additional input is just a shifted version of another one. These time-shifted basis functions are called redundant since they don't add new information and can be removed. Since they can just be removed, it should be noted that keeping them in, although the rank condition \cref{poe1} is not fulfilled, will still give the correct solution assuming removing them fulfills all conditions. A possible question that can be raised is what happens in other cases when rank condition \cref{poe1} is not fulfilled. If the rank is higher than required by the condition this often means that either the assumed rank is too low or \Cref{ass1} is not fulfilled. In the first case the method will work as intended assuming we have enough data according to the real order. In the second case the obtained solution will not be correct. How bad the result will be depends on how well the actual nonlinearity can be approximated by linear combinations of the basis functions. This shows the need for a correct choice of basis. If the rank is too low due to a lack of data limiting the number of columns of the matrix the story is a bit more complex. It is in most cases possible to obtain a solution that reproduces the reference output exactly using slightly less data than needed to fulfill the rank condition. The rank condition is needed for exact reproduction of the entire extended trajectory, but the output part of the trajectory might not depend on the entirety of the extended input trajectory. The terms the output doesn't depend on do not have to be reproduced exactly, meaning the system of equations needs less degrees of freedom to solve the output matching problem. Knowing how many extra samples one can do without requires knowledge about the terms that influence the output, which we don't have in general. An example regarding this will be given in a later section.

\section{Linear Parameter-Varying Systems}\label{sec4}
For this section we limit ourselves to the case of discrete-time linear parameter-varying systems with affine dependence on the parameters. Furthermore we assume the parameter values at each time instance are known. In that case the system can be written in the following form \cite{verhoek_fundamental_2021}:
\begin{equation}
    y(t)+\sum_{i=1}^{n_a}a_i(p(t-i))y(t-i)=\sum_{i=0}^{n_b}b_i(p(t-i))u(t-i)\label{system2}
\end{equation}
\begin{align*}
    a_i(p(t-i))&=\sum_{j=1}^{n_p}a_{i,j}p_j(t-i),\quad a_{i,j}\in\mathbb{R}^{n_y\times n_y}\\
    b_i(p(t-i))&=\sum_{j=1}^{n_p}b_{i,j}p_j(t-i),\quad b_{i,j}\in\mathbb{R}^{n_y\times n_u}
\end{align*}

This class of systems is already studied in the literature, for example in \cite{verhoek_fundamental_2021} and \cite{verhoek_data-driven_2021}. Knowledge about the parameter values over a prediction horizon is not always available. \cite{verhoek_data-driven_2021} offers some options and references to deal with this.

We now claim that discrete-time LPV systems with known parameters and affine dependence on those parameters are a specific case of output-generalized bilinear systems. Since $p(t)$ is a known signal we are in a similar situation as the \textit{LTI system with nonlinear term depending on other known signal} example. Ignoring the constant multiplication factor the product terms are of the form $p(t-i)\otimes y(t-i)$ and $p(t-i)\otimes u(t-i)$. Rewriting \cref{system2} into the form of \cref{system1} is possible by using definition \cref{OGB_phi_nl2} and taking
\begin{equation*}
    \phi_\mathrm{b}=\begin{bmatrix}
    p(t)^\top & \cdots & p(t-n_b)^\top
    \end{bmatrix}^\top,
\end{equation*}
\begin{equation*}
    \phi_\mathrm{b,0}=\begin{bmatrix}
    \mathbf{0}_{n_pn_u(n_b+1)\times 1} \\ p(t-1)\otimes y(t-1)\\ \vdots \\ p(t-n_a)\otimes y(t-n_a)
    \end{bmatrix}^\top,
\end{equation*}
taking $\theta_{\mathrm{lin}}=\mathbf{0}$ and
\begin{align*}
    \theta_\mathrm{nl}(i)=\big[
    &b_{0,1}(i, 1) \quad \cdots \quad b_{0,1}(i, n_u) \quad b_{0,2}(i, 1) \quad \\
    &\cdots \quad b_{0,n_p}(i, n_u) \quad b_{1,1}(i, 1)\quad\cdots \quad b_{n_b, n_p}(i, n_u)\\
    &-a_{1,1}(i, 1) \quad \cdots \quad -a_{1,1}(i, n_y) \quad -a_{1,2}(i, 1) \quad \cdots \quad \\
    &-a_{1,n_p}(i, n_y) \quad -a_{2,1}(i, 1)\quad\cdots \quad -a_{n_a, n_p}(i, n_y)
    \big]^\top_,
\end{align*}
with $\theta_\mathrm{nl}(i)$ a column vector with $n_p(n_u(n_b+1)+n_yn_a)$ rows and $\theta_\mathrm{nl}=\begin{bmatrix}
\theta_\mathrm{nl}(1) & \cdots & \theta_\mathrm{nl}(n_y)
\end{bmatrix}^\top$.
This shows that discrete-time LPV systems with known parameter values and affine dependence on these parameters are indeed output-generalized bilinear systems, allowing the results of the previous section to be applied. By taking $u_{\mathrm{nl}}(t)=\begin{bmatrix}
(p(t)\otimes u(t))^\top & (p(t)\otimes y(t))^\top\end{bmatrix}^\top$, where the redundant additional inputs got removed compared to the above result, we obtain the same system of equations in our method as found in \cite{verhoek_data-driven_2021} up to a permutation of the rows.

\section{Simulation examples}\label{sec5}
This section discusses the results obtained by simulating two nonlinear systems obtained by discretizing differential equations. For each system a data trajectory $w_{\mathrm{d}}$ of length $T$ is generated by applying a random input signal taken from a given distribution. This data trajectory is used to create the matrices needed by the method. Since we created the systems ourselves we know their properties, including the lag and order, which we use to verify if the rank condition \cref{poe1} is fulfilled. The quality of the result is quantified by applying the estimated input obtained by \Cref{theo1} and comparing the realized output obtained by applying the estimated input to the system with the reference output. The quantification is done using the realized relative mean-squared error (RRMSE) \begin{equation}
    \mathrm{RRMSE}(y_{\mathrm{realised}})=\frac{\norm{y_{\mathrm{realized}}-y_\mathrm{r}}_2}{\norm{y_\mathrm{r}}_2}.
\end{equation}

\subsection{Simulation example 1: Rotational pendulum}
This system is a discretized version of the rotational pendulum taken from \cite{verhoek_data-driven_2021}, where they show that this system can be rewritten as an LPV system. Here we will instead treat it like an output-generalized bilinear system. In continuous form the system is given by \begin{equation*}
    \frac{d^2 y}{dt^2}(t)=-\frac{mgl}{J}\text{sin}(y(t))-\frac{1}{\tau}\frac{dy}{dt}(t)+\frac{K_m}{\tau}u(t).
\end{equation*}
Here $y$ represents the angle of the pendulum. This system is discretized using a first order Euler approximation of the derivatives with sample time $T_s$. This gives a system of the form \begin{align}
    y(t)&=2y(t-1)-y(t-2)-\frac{T_s^2mlg}{J}\text{sin}(y(t-2))\nonumber\\
    &\quad-\frac{T_s}{\tau}(y(t-1)-y(t-2))+\frac{T_s^2K_m}{\tau}u(t-2)
\end{align}

It is clear to see that this system is of the form \cref{OGB_general} with $u_{\mathrm{nl}}=\text{sin}(y(t-2))$. We start off by generating a data trajectory $w_{\mathrm{d}}$ of the system using an input $u_{\mathrm{d}}\in\mathbb{R}^{307}$ with iid elements taken uniformly from the interval $]0,0.08[$. The needed matrices are calculated taking $L=100$ and $T_{\mathrm{ini}}=\textbf{n}(\mathcal{B}_{LTI})=2$. It is confirmed that the matrices fulfill rank condition \cref{poe1}. A reference trajectory is generated by starting the pendulum on $y=\pi/2$. We use inputs $u\in\mathbb{R}^{102}$ with the first two inputs zero and the rest iid taken uniformly from the interval $]0,0.05[$. The input estimation is done by applying \Cref{theo1} once. Its results can be seen in \Cref{fig:sim_ex2}.
\begin{figure}[ht]
\centering
\includegraphics[width=\linewidth]{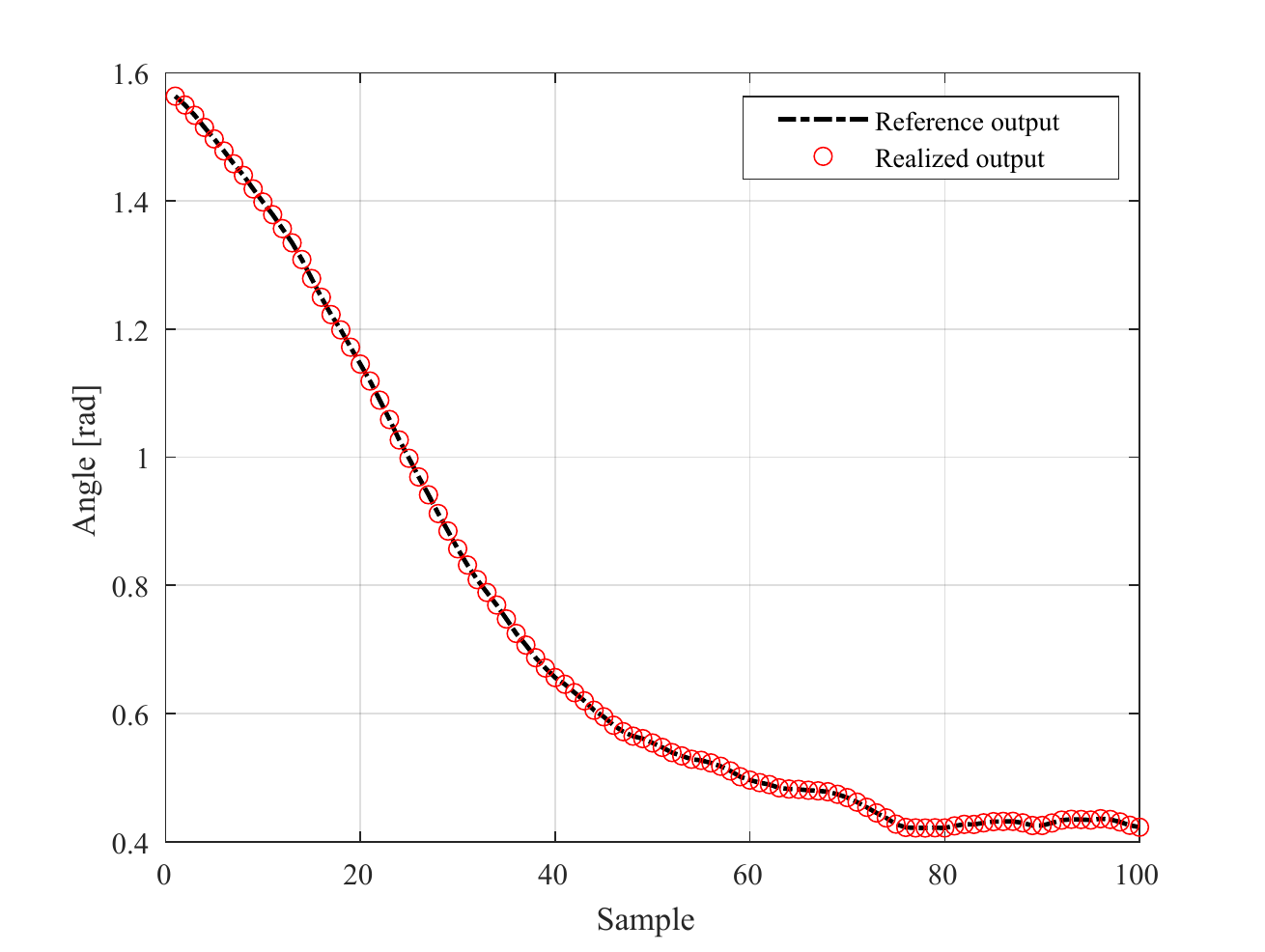}
\caption{Comparison between the reference output and the output realized by applying the estimated input for the rotational pendulum. Both are the same up to the machine precision.}
\label{fig:sim_ex2}
\end{figure}

If not enough data is available for this big a window it is also possible to use \Cref{theo1} using a smaller sliding window and applying the theorem multiple times for the different windows. We obtain a RRMSE of $4.806\cdot 10^{-14}$, which is around the machine precision. From \Cref{fig:sim_ex2} it can be seen that the output matching problem is perfectly solved. In this case the solution set can be parameterized using 2 parameters. As a general note, although not shown in this noiseless result, the noise performance of \Cref{theo1} is observed to be worse compared to the LTI case. A possible reason for this is that by applying nonlinear functions to the noisy data the implicit assumption that the noise is iid white Gaussian is no longer reasonable in most cases. Further exploring the influence of noise is left for future work.




\subsection{Simulation example 2: Four-tank system}
This example is taken from \cite{gatzke_model_2000}. It consists of four tanks with two tanks each flowing into the one below. Two pumps, one for each top tank, pump water from the bottom basin to the tank. We take the water level in each tank as the outputs and the pump speeds as the inputs. Discretized using a first order Euler approximation of the derivative with sample time $T_s$ we obtain following system:
\begin{align*}
  &y_1(t)=y_1(t-1)-T_s\Big(\frac{a_1}{A_1}\sqrt{2gy_1(t-1)}+\frac{a_3}{A_1}\sqrt{2gy_3(t-1)}\\
    &\quad+\frac{\gamma_1k_1}{A_1}u_1(t-1)\Big)\\
    &y_2(t)=y_2(t-1)-T_s\Big(\frac{a_2}{A_2}\sqrt{2gy_2(t-1)}+\frac{a_4}{A_2}\sqrt{2gy_4(t-1)}\\
    &\quad+\frac{\gamma_2k_2}{A_2}u_2(t-1)\Big)\\
    &\resizebox{\linewidth}{!}{$y_3(t)=y_3(t-1)-T_s\Big(\frac{a_3}{A_3}\sqrt{2gy_3(t-1)}+\frac{(1-\gamma_2)k_2}{A_3}u_2(t-1)\Big)$}\\
    &\resizebox{\linewidth}{!}{$y_4(t)=y_4(t-1)-T_s\Big(\frac{a_4}{A_4}\sqrt{2gy_4(t-1)}+\frac{(1-\gamma_1)k_1}{A_4}u_1(t-1)\Big).$}
\end{align*}
Here we take $u_{\mathrm{nl}}(t)=\sqrt{y(t)}$, applied to each element, which leads to 4 additional inputs. The parameter values used in the simulation are $a_i=2.3\mathrm{cm}^2$, $A_i=730\ \mathrm{cm}^2$, $k_1=5.51\ \mathrm{cm}^3/\mathrm{s}$, $k_2=6.58\ \mathrm{cm}^3/\mathrm{s}$, $g=981\ \mathrm{cm}/\mathrm{s}^2$, $\gamma_1=0.333$ and $\gamma_2=0.307$. We use inputs $u\in(\mathbb{R}^2)^{1410}$ with the first two inputs zero and the rest iid taken uniformly from the interval $]0,0.05[$. The input estimation is done by applying \Cref{theo1} once. We take $L=200$ and $T_\mathrm{ini}=2$. Taking a data trajectory length of 1410 is done because this is the minimal amount of data needed to fulfill rank condition \cref{poe1}. To illustrate the fact that you can get perfect reconstruction with a bit less in most cases we calculated the RRMSE value using different lengths of data. Looking at the system definition, it is expected that we can do with 6 samples less, corresponding to $u(L)$ and $u_{\mathrm{nl}}(L)$. The result can be seen in \Cref{fig:lack_of_data2}. 
\begin{figure}[ht]
\centering
\includegraphics[width=\linewidth]{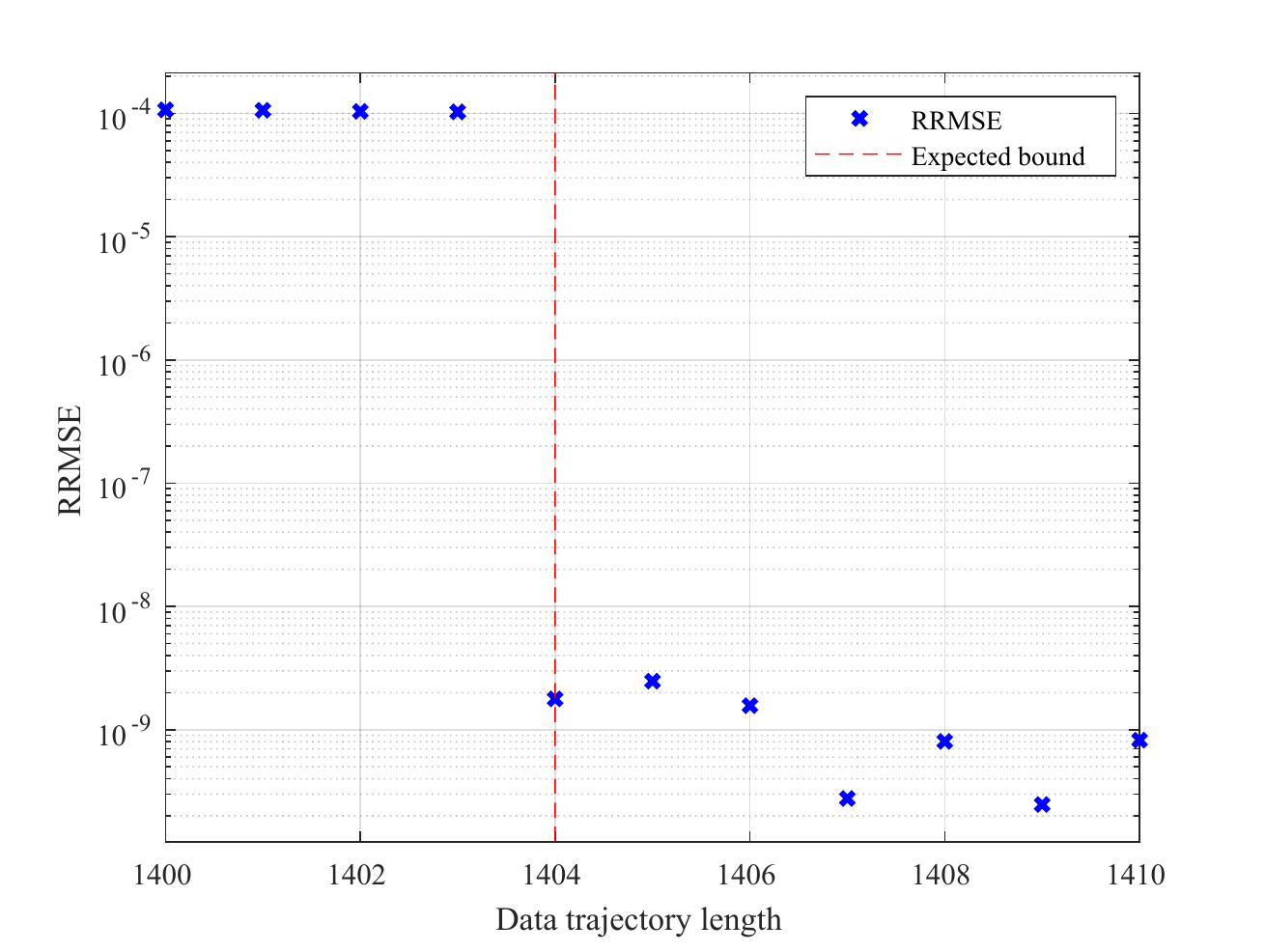}
\caption{Logarithmic plot of the RRMSE for different data trajectory lengths. Perfect reconstruction is predicted to last until a length of $1410-6=1404$.}
\label{fig:lack_of_data2}
\end{figure}

Note that if we do structure estimation by using the left null space of $\mathcal{H}_{\mathbf{l}+1}(w_{\mathrm{d,ext}})$ we need a trajectory of length $(n_\mathrm{u}+n_{\mathrm{nl}})(\mathbf{l}+1)+\mathbf{n}+\mathbf{l}=17$, which can be used to determine which terms influence the output. This can be used to remove unused basis functions and determine how many extra samples we can do without compared to those needed to fulfill the rank condition, allowing us to use the minimal amount of data.

\section{General discussion}\label{sec6}
\subsection{Non-uniqueness of the found solution}
\Cref{theo1} offers a way to obtain a solution to \Cref{prob1}, if one exists. In contrast to the simulation problem, however, this solution is not unique in general. One reason for this is due to the fact that reversing the input-output relationship does not guarantee that the resulting system is causal. The resulting system is non-causal in the case where the current output doesn't depend on the current input, but only on previous inputs. This non-uniqueness of the solution is not a problem, since this allows us to choose the most suitable option based on some criterion. To facilitate this, we will formulate a way to determine the minimal number of parameters needed for and how to construct a minimal parameterization of the solution set.

\begin{theorem}[Minimal parameterization of the solution set]
Consider the system of equations $Ag=b$ with solution set $\mathcal{G}$ and a matrix $U$ with the same number of columns as $A$. If $\mathcal{G}$ is non-empty, the set $\mathcal{U}=\{Ug\ |\ g\in\mathcal{G}\}$ can be parameterized using at least $\rank UN$ parameters, with $N$ a basis for the null space of $A$.\label{param}
\end{theorem}
\begin{proof}
In the case where $\mathcal{G}$ has only one element, $\nul A=\{0\}$. This means that $\rank UN=0$. By definition $\mathcal{U}$ also has only one element, which means the set can indeed be parameterized using zero parameters.
    
In the other case $\mathcal{G}$ has more than one element, meaning $A$ is not full column rank. Writing the number of columns of $N$ as $n_1$, this means the solution set $\mathcal{G}$ can be parameterized as \begin{equation}
    \mathcal{G}=\{A^\dagger b+Nz\ |\ z\in\mathbb{R}^{n_1}\}.\nonumber
\end{equation}

Making use of this we can parameterize $\mathcal{U}$ as follows:\begin{equation}
    \mathcal{U}=\{UA^\dagger b+UNz\ |\ z\in\mathbb{R}^{n_1}\}.\nonumber
\end{equation}

This parameterization is in general not minimal. A minimal parameterization requires that $UN$ has full column rank. From this we can immediately see that the minimal number of parameters is given by $\rank UN$

\end{proof}

\subsection{Required amount of data}
As a further comparison between the found results and the LTI case we will take a look at the minimal amount of data needed in advance to solve the output matching problem. We will focus on the minimal needed trajectory length of $w_{\mathrm{d}}$. The minimal trajectory length is based on the amount of columns needed to fulfill rank condition \cref{poe1}. The relationship between the number of columns $j$ of a block Hankel matrix of a trajectory with length $T$ that has $L$ block rows is $L+j-1=T$.

In the case of output-generalized bilinear systems we have following condition:
\begin{equation}
    \rank\mathcal{H}_{L+T_{\mathrm{ini}}}(w_{\mathrm{d,ext}})=(n_u+n_{\mathrm{nl}})(L+T_{\mathrm{ini}})+\textbf{n}(\mathcal{B}).\nonumber
\end{equation}
Due to the construction of $u_{\mathrm{nl}}$ used in $w_{\mathrm{d,ext}}$, there are $T_{\mathrm{ini}}$ samples we cannot use. These have to be taken into account. We obtain:
\begin{align*}
    T-T_{\mathrm{ini}}&=L+T_{\mathrm{ini}}+j-1\\
    \Rightarrow j&=T-L+1-2T_{\mathrm{ini}}=(n_u+n_{\mathrm{nl}})(L+T_{\mathrm{ini}})+\textbf{n}(\mathcal{B})\\
    \Rightarrow T&=(n_u+n_{\mathrm{nl}}+1)(L+T_{\mathrm{ini}})+\textbf{n}(\mathcal{B})+T_{\mathrm{ini}}-1.
\end{align*}

For an LTI system with the same $\textbf{n}$ we obtain:
\begin{align*} j&=T-L-T_{\mathrm{ini}}+1=n_u(L+T_{\mathrm{ini}})+\textbf{n}(\mathcal{B})\\
    \Rightarrow T&=(n_u+1)(L+T_{\mathrm{ini}})+\textbf{n}(\mathcal{B})-1.
\end{align*}

From this we can see that we need $n_{\mathrm{nl}}(L+T_{\mathrm{ini}})+T_{\mathrm{ini}}$ extra samples compared to the LTI case. If we take the situation $\textbf{n}(\mathcal{B})=T_{\mathrm{ini}}=2$, $L=10$, $n_u=1$ and $n_{\mathrm{nl}}=3$ we obtain a minimal trajectory length of $T=5\cdot 12+3=63$ for the output-generalized bilinear system and $T=2\cdot 12+1=25$ for the LTI case, with indeed a difference of $3\cdot 12+2=38$ samples. Note that this result is for the most general form of this model class. It assumes that to calculate $u_{\mathrm{nl}}$ we need $T_{\mathrm{ini}}$ samples further in the past, counting from the initial trajectory. In the case that $u_{\mathrm{nl}}$ doesn't depend on samples that far back you need less than $T_{\mathrm{ini}}$ extra samples, up to zero extra samples when $u_{\mathrm{nl}}$ only depends on current inputs and outputs.

\section{Conclusion and future work}\label{sec7}
This work has shown the feasibility of using data-driven simulation techniques to solve the output matching problem. It turns out that no constraints have to be added compared to the simulation case. The derived methods offer a way to parameterize the solution set, allowing the end user to choose the solution that best fits their criteria. It is shown that linear parameter-varying systems are a specific case of output-generalized bilinear systems, allowing the results of this work to also be applied to this class of systems.

Preliminary work has been done, but improvements are still possible. The next step is making the proposed method more robust to measurement noise on the trajectories. After that the feasibility of adapting this method for the reference tracking control problem in the general case has to be analyzed. Some other possible improvements are relaxing the bijectivity condition on transformation $h$ to surjectivity, which should suffice for a lot of cases and analyzing the noise distributions arising from popular choices of nonlinear basis functions. Something else to look at is the complexity of solving the structure estimation problem. During this entire work we have made the assumption that the nonlinear transformations are known up to a finite set of basis functions, but even limiting yourself to up to third order monomials can lead to a large number of basis functions for an increasing number of inputs and outputs. Better quantifying this complexity is worth looking into.